\documentclass[11pt]{amsart}

\usepackage{hyperref}
\usepackage{amssymb,amsfonts}
\usepackage{hyperref}
\usepackage{amssymb,textcomp}
\usepackage{enumerate}

\usepackage[utf8]{inputenc}
\usepackage[T1]{fontenc}

\usepackage[mathscr]{eucal}

\newcommand{\N}{{\mathbb N}}

\newcommand{\R}{{\mathbb R}}
\newcommand{\T}{{\mathbb T}}
\newcommand{\Z}{{\mathbb Z}}
\newcommand{\ve}{\varepsilon}

\newcommand{\norm}[1]{\left\Vert #1\right\Vert}
\newcommand{\nnorm}[1]{\lvert\!|\!| #1|\!|\!\rvert}

\newcommand{\inv}{^{-1}}

\newtheorem{lemma}{Lemma}
\newtheorem{proposition}[lemma]{Proposition}
\newtheorem*{proposition*}{Proposition}
\newtheorem{theorem}[lemma]{Theorem}
\newtheorem*{theorem*}{Theorem}
\newtheorem{corollary}[lemma]{Corollary}
\newtheorem*{corollary*}{Corollary}

\theoremstyle{definition}
\newtheorem*{definition}{Definition}

\theoremstyle{remark}
\newtheorem*{remark}{Remark}

\theoremstyle{plain}

\begin{document}
\title[Bohr recurrence and density of non-lacunary semigroups]{ Bohr recurrence and density of non-lacunary semigroups of $\N$ }

\thanks{NF was supported  by the research grant  ELIDEK HFRI-NextGenerationEU-15689 and  BK was partially supported by the National Science Foundation grant DMS-2348315.}

\author{Nikos Frantzikinakis}
\address[Nikos Frantzikinakis]{University of Crete, Department of mathematics and applied mathematics, Voutes University Campus, Heraklion 71003, Greece} \email{frantzikinakis@gmail.com}
\author{Bernard Host}
\address[Bernard Host]{
LAMA, Universit\'e Gustave-Eiffel, CNRS, F-77447 Marne-la-Vall\'ee, France}
\email{bernard.host@univ-eiffel.fr}
\author{Bryna Kra}
\address[Bryna Kra]{Department of Mathematics, Northwestern University, Evanston, IL 60208, USA} \email{kra@math.northwestern.edu}

\begin{abstract}
A subset $R$ of integers is a set of Bohr recurrence if every rotation on $\T^d$ returns arbitrarily close to zero under some non-zero multiple of $R$. We show that the set $\{k!\, 2^m3^n\colon k,m,n\in \N\}$ is a set of Bohr recurrence. This is a particular case of a more general statement about images of such sets under any integer polynomial with zero constant term. We also show that if $P$ is a real polynomial with at least one non-constant irrational coefficient, then the set 
$\{P(2^m3^n)\colon m,n\in \N\}$ is dense in $\T$, thus providing a joint generalization of two well-known results, one of Furstenberg and one of Weyl. 
\end{abstract}


\maketitle

\section{Introduction and main results}
\subsection{Notions of recurrence}
The recurrence properties of a dynamical system are a classical way to study the system's qualitative and quantitative properties. We focus on recurrence in a topological dynamical system $(X, T)$, meaning that $X$ is a compact metric space and $T\colon X\to X$ is a homeomorphism, and further assume that the system is minimal, meaning that no proper closed set in $X$ is $T$-invariant.  A set $R\subset\N$ is a set of \emph{(topological) recurrence} if for every minimal system $(X, T)$ and non-empty open set $U\subset X$, there is some $r\in R$ such that $U\cap T^{-r}U\neq\emptyset$. It follows quickly  from the compactness of $X$ that  $\N$ is a set of recurrence, and many sparser subsets have been shown to be sets of recurrence.  These include, for example, the set of differences of an infinite set (see~\cite{Furstenberg-book} for this example and others), the set $\{P(n):n\in\N\}$ where $P$ is any integer valued polynomial with no constant term (see~\cite{F2, sarkozy1}), and the set of shifted primes $\{p-1: p \text{ is a prime}\}$ (see~\cite{sarkozy2}).  It is also easy to check that parity obstructions give rise to examples, such as the odd numbers or any shift of the primes other than $\pm 1$, that cannot be sets of recurrence.  

A natural question is if restricting the class of topological systems changes the possible sets of recurrence.  The most basic nontrivial systems to consider are  rotations on the  $d$-dimensional torus $\T^d = \R^d/\Z^d$.  For 
 $x\in \R^d$, write $\norm{x}_{\T^d}:=d(x,\Z^d)$, and we define sets of recurrence for these systems. 
\begin{definition}
A subset $R$ of the integers is a \emph{set of Bohr recurrence} if for every integer $d\in \N$, all $\alpha_1,\dots,\alpha_d\in\R$, and every $\ve>0$, there exists non-zero $r\in R$ such that
$$
\norm{ r\alpha_j}_\T\leq \ve \quad \text{for }j=1,\dots,d. 
$$
\end{definition}
Equivalently, the set $R$ is a set of Bohr recurrence if for every integer $d\in \N$ and every $x\in\T^d$, the
point $0$ belongs to the closure of  the set $\{ rx\colon r\in R\setminus\{0\}\}$ in $\T^d$. 
By building an appropriate rotation, one can check that no lacunary set can be a set of Bohr recurrence, and therefore cannot be a set of recurrence. 
An important open problem,  popularized by Katznelson~\cite{K}, is  whether a set of Bohr recurrence is necessarily a set of topological recurrence.  See \cite{We, BG, HKM, GKR,Gr} for various equivalent formulations of this question and related results.

\subsection{Non-lacunary semigroups}
Our objective is to derive recurrence and density properties of sets generated by thin non-lacunary semigroups of the integers, a prototypical example being the set
$
\{ 2^m3^n \colon m,n\in\N \}. 
$
Again, parity reasons prevent this from being a set of Bohr recurrence, but after  removing this obstruction we are led to the well-known question in dynamics
 whether the set  
$$
\{ k!\,2^m3^n \colon k,m,n\in\N \}
$$
 is a set  of  topological recurrence (more generally one can ask if this is a set of measurable multiple recurrence). Although we are unable to answer these  questions,
we derive some  recurrence  results of intermediate strength, which follow from Furstenberg's Diophantine Theorem~(see Section~\ref{SS:DensityFW}), using elementary manipulations.

Two positive integers are \emph{multiplicatively independent} if their only common power is $1$. The first main goal of this note is to establish the following result:
\begin{theorem}
	\label{T:recurrence}
Let  $p,q\in \N$ be  multiplicatively independent integers. Then the set $\{ k!\,p^mq^n \colon k,m,n\in\N \}$ is a set of Bohr recurrence.
\end{theorem}
More generally,  it follows from our argument that $k!$ can be replaced by any sequence that contains multiples of every positive integer.

On the other hand, considering rational multiples of elements of this set shows that the multiplicative factor $k!$ cannot be removed, and, in fact, it is necessary for our argument to work even if  the real numbers $\alpha_1,\ldots, \alpha_d$ have no rational relations.
As shown in~\cite[Theorem~2]{M}, for $d\geq 4$, there exists $x\in \T^d$ with rationally independent coordinates for which the set   $\{2^m3^n x \colon m,n\in\N\}$ is not dense in $\T^d$.   Although  for fixed $d\in \N$ this particular obstruction 
can  be removed if
more powers are included in the set (see~\cite{Y} for a density statement  on $\T^2$ that uses three powers),  \cite[Theorem~2]{M} shows that the problem persists even when using $\ell$ powers instead of two and taking $d$ large enough;  in this case we have non-density in $\T^d$ for some  $x\in \T^d$ with rationally independent coordinates    whenever $d\geq 2\ell$. Furthermore, by slightly modifying the example in~\cite[Lemma~4.2]{M}, we obtain a point $x\in \T^4$ with rationally independent coordinates, such that for some  $\varepsilon>0$ we have  $\norm{2^m3^nx}_{\T^4}\geq \varepsilon$ for every $m,n\in \N$ (see Section~\ref{S:NonBohr}). 

We also remark that combining Theorem~\ref{T:recurrence} with~\cite[Theorem~4.1]{HKM} and~\cite[Theorem~A]{GKR}, we deduce that the sets in Theorem~\ref{T:recurrence} (and in Theorem~\ref{T:polynomialrecurrence} below) are also good for topological recurrence for all  nilsystems and for a large class of skew product  systems.

 In the proof of Theorem~\ref{T:recurrence}, we consider a more general situation where sets of the form $\{p^mq^n\colon m,n\in \N\}$ are replaced by more general non-lacunary semigroups of $\N$.
\begin{definition}
A \emph{semigroup of $\N$}  is a  sub-semigroup of $(\N,\cdot)$ and it is \emph{non-lacunary} if it contains two multiplicatively independent integers.
\end{definition}
Equivalently, a semigroup is non-lacunary if it is not contained  in the set of powers of a single integer.
It is shown in~\cite[Lemma~IV.1]{F} that a semigroup $\Sigma=\{s_1<s_2<\cdots \}$ of $\N$ is non-lacunary if and only if $s_{n+1}/s_n\to 1$ as $n\to \infty$.

We say that $P$ is an \emph{integer polynomial} if it  has integer coefficients. We prove the following result, generalizing Theorem~\ref{T:recurrence}.
\begin{theorem}
	\label{T:polynomialrecurrence}
	Let $\Sigma$ be a non-lacunary semigroup of  $\N$, $K\subset \N$ be a set that contains multiples of every positive integer,
and $P$ be a non-constant  integer polynomial with $P(0)=0$. Then the set $\{ P(k s) \colon k\in K, s\in \Sigma\bigr\}$ is a set of  Bohr recurrence.
\end{theorem}
In particular, if $p,q\in \N$ are multiplicatively independent and $P$ is a  non-constant integer polynomial with  zero constant term, then   the set  $$
\{ P(k!\,p^mq^n) \colon k,m,n\in\N\bigr\}
$$ is a set of Bohr recurrence.
 By taking $P(n)=n$, Theorem~\ref{T:recurrence} follows.

 The
argument used to prove Theorem~\ref{T:polynomialrecurrence} can also be used, essentially without change,  to show that if $P_1,\ldots, P_\ell$ are integer polynomials with zero constant terms, then for every $\alpha_1,\ldots, \alpha_\ell\in \R$ and $\ve>0$,  there exist $k\in K$  and  $s\in \Sigma$ such that $\norm{P_j(k\cdot s)\, \alpha_j}_\T\leq \ve$ for $j=1,\ldots,\ell$.

We prove  Theorem~\ref{T:polynomialrecurrence} in Section~\ref{S:ProofBohr}. The  main   ingredient is Proposition~\ref{prop:periodic}, which is of independent interest and  shows  the existence of a point with rational coordinates in the orbit closure of non-lacunary semigroup actions of $\T^d$. A special case of this result  is that if $\Sigma$ is a non-lacunary semigroup of $\N$ and $A$ is a non-empty closed subset of $\T^d$, which  is invariant under coordinate-wise multiplication by 
elements of the form $(s,\ldots, s)$ for every $s\in \Sigma$, then $A$ contains a point with rational coordinates. Theorem~\ref{T:recurrence} is a direct consequence of this result.

\subsection{A Weyl-type extension of Furstenberg's result}\label{SS:DensityFW}
We recall a celebrated theorem of Furstenberg~\cite[Theorem~IV.1]{F} (see also \cite{Bos} for an elementary proof):  
\begin{theorem*}[Furstenberg]
	\label{th:F}
	Let $\Sigma$ be a non-lacunary semigroup of $\N$ and $\alpha$ be irrational. Then the set 
		$
	\{s\alpha \colon s\in \Sigma\}
	$
	is dense in  $\T$.  
\end{theorem*} 
This immediately implies that a closed $\Sigma$-invariant subset of $\T$ is either dense or finite. In the latter case, it consists entirely of rational points. 

Another celebrated theorem, this time  of Weyl~\cite{W16}, is as follows: 
\begin{theorem*}[Weyl]
	\label{th:W}
 Let $P\in \R[t]$ be a   polynomial with at least one  non-constant irrational coefficient.   Then the set 
$
\{P(n) \colon n\in \N\}
$
is dense in $\T$. 
\end{theorem*}

We prove the following common generalization of these two density statements:  
\begin{theorem}
	\label{T:dense}
	Let   $\Sigma$ be a non-lacunary semigroup and $P\in \R[t]$ be a   polynomial with at least one   non-constant irrational coefficient.   Then the set 
	$$
	\{P(s) \colon s\in \Sigma\}
	$$
	is dense in $\T$. 
\end{theorem}
It should be noted that Theorem~\ref{T:dense}  follows from \cite[Theorem~1.2]{Kra}, but the argument presented here is somewhat different and more detailed. 

Taking $P(n):=n$ gives Furstenberg's result and taking $\Sigma:=\N$ gives Weyl's result.  On the other hand, even for the simplest non-linear cases, such as $P(n)=(n^2+n)\alpha$, where $\alpha$ is irrational, it is not clear how to proceed and prove density, since   the set $\{P(s) \colon s\in \Sigma\}$
does not  seem to satisfy any useful $\Sigma$-invariance property. A key maneuver in our argument is to instead work with the subset 
	$\{(s\alpha,s^2\alpha) \colon s\in \Sigma\}$ of $\T^2$, which is 
	invariant under multiplication by all elements of the set $\{(s,s^2)\colon s\in \Sigma\}$. 
	 In Proposition~\ref{P:densehigher}, we show that this invariance implies that the closure of this set contains a  line segment parallel to some coordinate axis. Theorem~\ref{T:dense} follows easily from this fact. 

Finally, we  note that more general $d$-dimensional versions of Theorem~\ref{T:dense} fail.
For example,  if $P_1,\ldots, P_d\in \R[t]$ are such that every non-trivial integer combination of the polynomials is a polynomial with at least one irrational non-constant coefficient, and $\Sigma$ is a non-lacunary semigroup of $\N$,  it is not always true that 
the set 
$$
\{(P_1(s),\ldots, P_d(s))\colon s\in \Sigma\}
$$
is dense in $\T^d$ (this holds for $\Sigma:=\N$ since equidistribution follows from Weyl's criterion). To see this, take $d=4$ and  
$\Sigma:=\{2^m3^n\colon m,n\in\N\}$, $P_j(n)=n\alpha_j$, $j=1,\ldots, 4$,   for the rationally independent 
reals $\alpha_1,\ldots, \alpha_4$ constructed in \cite[Lemma~4.2]{M}, so that the above set is not dense. 

\subsection{Notation} \label{SS:notation}
We denote the set of positive integers by $\N$.  We let $\T$ 
denote the one dimensional torus  $\R/\Z$, and we often identify it   with $[0,1)$.
We denote elements of $\T$ by real numbers and we  implicitly  assume that these real numbers are taken  modulo $1$.
 
If $ (t_1,\ldots, t_d)$ and $(x_1,\ldots, x_d)$ are two vectors in $\R^d$, we define their product by coordinate-wise multiplication as follows
$$
(t_1,\ldots, t_d)\cdot(x_1,\ldots, x_d):=(t_1x_1,\ldots, t_dx_d).
$$
Likewise,  if  $(k_1,\ldots, k_d)\in \Z^d$ and $(x_1,\ldots, x_d)\in \T^d$,   we write 
$$
(k_1,\ldots, k_d)\cdot (x_1,\ldots, x_d):=(k_1x_1,\ldots, k_dx_d).
$$
If $S$ is  a subset of $\Z^d$ and $A$ a subset of $\T^d$, we say that $A$ is {\em $S$-invariant} or {\em invariant under $S$} if $s \cdot x\in A$ for every $s\in S$, $x\in A$.  

\subsection{Acknowledgement} The authors would like to thank J.~Griesmer and R.~Alweiss whose questions related to the proof of  Theorem~\ref{T:recurrence} motivated the authors to write this manuscript. 


\section{Existence of rational points}\label{S:Rational}
 Furstenberg's Theorem implies  that if $\Sigma$ is a non-lacunary semigroup of $\N$, then a closed $\Sigma$-invariant subset of $\T$ always contains a rational point.  A natural question to consider is whether there are generalizations of this statement in higher dimensions. We prove that this is the case, and  use this result in the proofs of both  Theorems~\ref{T:polynomialrecurrence} and ~\ref{T:dense}. 
To give the precise statement, we need another definition.
\begin{definition}
 A point $x\in \T^d$ is \emph{rational} if its coordinates are rational. If $\Sigma$ is a subset of $\N$,   we say that the \emph{denominators of  a rational point $x$ are  relatively prime to $\Sigma$} if its coordinates are rationals with denominators relatively prime to each element of $\Sigma$. By convention, the denominator of $0$ is $1$.
\end{definition}
\begin{proposition}
	\label{prop:periodic}
	Let $\Sigma$ be a non-lacunary semigroup, $d\in \N$, and $\ell_1,\ldots, \ell_d\in \N$ (not necessarily distinct).  Let   $A$ be a non-empty closed subset of $\T^d$,  invariant under 
	$$
	\Sigma^{\ell_1,\ldots, \ell_d}:=\{(s^{\ell_1},\ldots, s^{\ell_d}):  s\in \Sigma\}.
	$$
	 Then $A$ contains a rational point with denominators relatively prime to   $\Sigma$.
\end{proposition}
\begin{remark}
	For the proof of Theorem~\ref{T:recurrence}, only the case $\ell_1=\cdots=\ell_d=1$ is needed. The more general case is needed in the proofs of Theorems~\ref{T:polynomialrecurrence} and \ref{T:dense}.
\end{remark}
\begin{proof}
We proceed by induction on $d\in \N$.  Throughout, for  $\ell\in \N$ we let
	$$
	\Sigma^{\ell}:=\{s^{\ell}:  s\in \Sigma\}. 
	$$
	\subsubsection*{The case $d=1$} Since $\Sigma^{\ell_1}$ is a non-lacunary semigroup of $\N$,  by  Furstenberg's Theorem (see Section~\ref{SS:DensityFW}) we have that $A$ contains a rational point $x=a/b$ with $a\in\Z$ and $b\in\N$. Write $b=b'c$, where $b'$ divides some element $k$ of $\Sigma$ and $c$ is relatively prime to every element of $\Sigma$. Then $kx = \frac{a(k/b')} c$ and thus this point of $A$ is rational with  denominator  relatively prime to $\Sigma$.
		
	\subsubsection*{Suppose that the result holds for some  $d\in \N$} We show that it  holds for $d+1$.  Let $A$ be a non-empty closed subset of $\T^{d+1}$, invariant under $\Sigma^{\ell_1,\ldots, \ell_{d+1}}$, where $\ell_1,\ldots, \ell_{d+1}\in \N$.
	Let $B$ be the image of $A$ under the coordinate projection $$
	(x_1,\dots,x_d,x_{d+1})\mapsto x_{d+1}
	$$ of $\T^{d+1}$ onto $\T$. Then $B$ is a non-empty closed subset of $\T$, invariant under the non-lacunary semigroup  $\Sigma^{\ell_{d+1}}$, and hence by the case  $d=1$, 
	it contains a rational point  $y$ with denominator  relatively prime to $\Sigma$.
	
	Write
	$y=a/m$, where $a\in\Z$  and $m$ is a positive integer relatively prime to every element of $\Sigma$, and let
	$$
	\Sigma_m:=\Sigma\cap(m\Z+1).
	$$
	Then $\Sigma_m$ is a semigroup of $\N$ with the property that
	$$
	sy=y \quad \text{for every }s\in \Sigma_m.
	$$
	We claim that $\Sigma_m$ is non-lacunary. Indeed, if $p$ and $q$ are two multiplicatively independent integers belonging to $\Sigma$, then $p$ and $q$ are relatively prime to $m$ (because  the denominator of $y=a/m$ is relatively prime to $\Sigma$).  
	Hence,  by Euler's Theorem, there exist  integers $t$  and $n$ such that $p^t\equiv 1\pmod m$ and $q^n\equiv 1\pmod m$. Then $p^t$ and $q^n$ are multiplicatively independent and belong to $\Sigma_m$ and thus the semigroup $\Sigma_m$  is non-lacunary.

	We define the set 
	$$
	C:=\bigl\{(x_1,\dots,x_d)\in\T^d\colon(x_1,\dots,x_d,y)\in A\bigr\}.
	$$
	Then $C$ is a closed subset of $\T^d$,  and is non-empty since $y\in B$.  We claim that $C$ is invariant under the semigroup
	$$
	\Sigma_m^{\ell_1,\ldots, \ell_d}:=\{(s^{\ell_1},\ldots, s^{\ell_d}): s\in \Sigma_m\}.
	$$
	 Indeed,  if
	$(x_1,\dots,x_d)\in C$, using that  $s^{\ell_{d+1}}y=y$ for every $s\in \Sigma_m$ and the invariance of $A$ 
	 under $\Sigma^{\ell_1,\ldots, \ell_{d+1}}$, we have that 
	$$
	(s^{\ell_1}x_1,\dots,s^{\ell_d}x_d,y)=(s^{\ell_1}x_1,\dots,s^{\ell_d}x_d,s^{\ell_{d+1}}y)
	\in A \quad \text{for every }s\in \Sigma_m. 
	$$
	 Using the defining property of $C$, we deduce that    $(s^{\ell_1}x_1,\dots,s^{\ell_d}x_d)\in C$.
	
Using the induction hypothesis with $C$ substituted for $A$ and $\Sigma_m^{\ell_1,\ldots, \ell_d}$ substituted for $\Sigma^{\ell_1,\ldots, \ell_d}$, we deduce that $C$ contains a rational point $(y_1,\dots,y_d)$  with denominators   relatively prime to $\Sigma$. Then the point $(y_1,\dots,y_d,y)$ is rational with denominators relatively prime to  $\Sigma$, and belongs to $A$, as desired. 
\end{proof}

\section{Proof of  the Bohr-recurrence result}\label{S:ProofBohr}
\subsection{Proof of Theorem~\ref{T:polynomialrecurrence}}
Suppose that
$$
P(n)=c_1n+c_2n^2+\cdots+c_rn^r
$$ for some $r\in \N$ and
$c_1,\ldots, c_r\in \Z$,  not all of them zero. Let $\ve>0$, $d\in \N$,  and $x=(x_1,\ldots, x_d)\in \T^d$.  Let $A$ be the closure of the set
$$
\{(sc_1x_1,\dots,sc_1x_d,
s^2c_2x_1,\dots,s^2c_2x_d,
\dots,
s^rc_rx_1,\dots,s^rc_rx_d):  s\in \Sigma\}.
$$
Then $A$ is a non-empty closed subset of $\T^{rd}$ that is invariant under 
$$
\{(s,\ldots, s, s^2,\ldots, s^2,\ldots, s^r,\ldots, s^r): s\in \Sigma\},
$$
where each power of $s$ is repeated $d$ times.  Therefore, by  Proposition~\ref{prop:periodic}, $A$ contains a rational point $y\in \T^{rd}$.
Since $K$ contains  multiples of every positive integer, there exists $k\in K$ such that all the coordinates of $ky$ are integers. It follows that $0$ (as an element of $\T^{rd}$) belongs to the closure of the set
$$
\{(ksc_1x,(ks)^2c_2x,\ldots, (ks)^rc_r x): s\in \Sigma\}.
$$
Since
 $
 P(ks)\, x=ksc_1x+(ks)^2c_2x+\cdots+(ks)^rc_rx,
 $
 we have that 
 $$
 \norm{P(ks)\, x}_{\T^d}\leq \norm{ksc_1x}_{\T^d}+\norm{(ks)^2c_2x}_{\T^d}+\cdots+\norm{(ks)^rc_rx}_{\T^d}. 
 $$
It follows that $0$ (as an element of $\T^d$) belongs to the closure of the set
$$
\{ P(ks) \, x\colon k\in K: s\in \Sigma\bigr\}
$$
and so
\begin{equation}\label{E:Pks}
\norm{P(ks)\, x}_{\T^d}\leq\ve, \quad  \text{for some } k\in K, s\in \Sigma.
\end{equation}
The polynomial $P$ has finitely many zeros, all bounded by some integer $N\in\N$.
Estimate \eqref{E:Pks} also holds for the semigroup $\Sigma\cap(N,+\infty)$  instead of $\Sigma$, and thus replacing $\Sigma$ with this semigroup
 we deduce that  there exist $k\in K$ and $s\in \Sigma$ such that $P(ks)\neq 0$ and
 $\norm{P(ks)\, x}_{\T^d}\leq\ve$.  It follows that the set $\{ P(k s) \colon k\in K, s\in \Sigma\bigr\}$ is a set of  Bohr recurrence. 
\qed

\subsection{Non Bohr recurrence of $2^m3^n$ for totally ergodic  rotations}\label{S:NonBohr}
Although the set  $\{2^m3^n\colon m,n\in \N\}$ is not good for recurrence of rational rotations, it is reasonable to hope that it is good for recurrence for all $x\in \T^d$ that have rationally independent coordinates.
An even more optimistic conjecture is that it   is good for measurable recurrence for all totally ergodic systems. Unfortunately, we show that neither  claim holds. 

We claim that there exists  $x\in \T^4$ with rationally independent coordinates  and   $\varepsilon>0$ such that  $\norm{2^m3^nx}_{\T^4}\geq \varepsilon$ for every $m,n\in \N$.
Indeed,  using~\cite[Lemma~4.2]{M}, we have the existence of 
$y=(y_1,y_2,y_3,y_4)\in \R^4$ with rationally independent coordinates such  that for all but finitely many $m,n\in \N$ we have
 	$\{2^m3^ny_j\}\leq 1/10$
for some $j\in \{1,2,3,4\}$, where $\{t\}$ denotes the fractional part of $t\in \R$.
 	Let $x_j:=y_j+1/5$, $j\in
 	\{1,2,3,4\}$, and $x=(x_1,x_2,x_3,x_4)$; then obviously $x$ also has rationally independent coordinates. Since $\{2^m3^n/5\}\in [1/5,4/5]$
	for all $m,n\in \N$, it follows that for all but finitely many $m,n\in \N$ we have $\{2^m3^nx_j\}\in [1/5,9/10]$ for
 	some $j\in \{1,2,3,4\}$. Then $\norm{2^m3^nx}_{\T^4}\geq 1/10$ for all but finitely many $m,n\in \N$.
 	This immediately implies the claim.

 The situation does not improve much if we consider  sets of the form
 $$
 \Sigma:=\{p_1^{n_1}\cdots p_\ell^{n_\ell}\colon n_1,\ldots, n_\ell\in \N\}.
 $$
  Using~\cite[Lemma~4.2]{M}, we can show in a similar fashion that there exists $x\in \T^{2\ell}$ with rationally independent coordinates and $\varepsilon>0$ such that $\norm{sx}_{\T^{2\ell}}\geq \ve$ for every $s\in \Sigma$.

  However, the possibility that  the set  $\{2^m3^n\colon m,n\in \N\}$ is good for  measurable or topological  recurrence for
  systems without rotational factors (i.e. for weakly mixing systems) remains open. For measurable recurrence this problem is explicitly stated in~\cite[Question~3]{B}.

 \section{Proof of  the density result}
 Our goal in this section is to prove Theorem~\ref{T:dense}.

\subsection{A reduction}
Theorem~\ref{T:dense} is a direct consequence of the following higher dimensional density result. This higher dimensional setting has the advantage of some dilation invariance that the $1$-dimensional setting lacks. 

Let $e_1,\ldots, e_d$ denote the standard unit vectors of $\T^d$ or  $\R^d$, depending on the context.
\begin{proposition}
	\label{P:densehigher}
	Let $d\in \N$,  $\Sigma$ be a non-lacunary semigroup of $\N$, and 
	$A$ be a closed 
	infinite subset of $\T^d$ that is invariant under   
	$$
	\{(s,s^2,\ldots, s^d)\colon s\in \Sigma\}.
	$$   
	Then 
	$
	\{x+te_j\colon t\in \T\}\subset A
	$
	for some   rational point $x\in \T^d$ and  $j\in \{1,\ldots, d\}$. 
\end{proposition}
\begin{proof}[Proof of Theorem~\ref{T:dense} assuming Proposition~\ref{P:densehigher}]
	Let $\varepsilon>0$ and $\alpha \in \T$ be arbitrary.  It suffices to show that there exists $s\in\Sigma$ such that $\norm{P(s)-\alpha}_\T\leq \ve$. 
	
	We can assume that $P(0)=0$ and write $P(n)=c_1n+\cdots+c_dn^d$ for some $d\in \N$ and real numbers  $c_1,\ldots, c_d$, at least one of which is irrational. 
	Let  $A$ be the closure of the set 
	$$
	\{(c_1s,c_2s^2,\ldots, c_ds^d)\colon n\in \Sigma\},
	$$ 
	considered as a subset of $\T^d$. Clearly $A$ is invariant under the map  $x\mapsto (s,s^2,\ldots, s^d)  \cdot x$ for every $s\in \Sigma$, and since at least one of  the $c_1,\ldots, c_d$ is irrational the set $A$ is infinite.  It follows from Proposition~\ref{P:densehigher} that  there exist some rational point $x=(x_1,\ldots, x_d)\in \T^d$ and  $j_0\in \{1,\ldots, d\}$ such that  $\{x+te_{j_0}\colon t\in \T\}\subset A$. Hence, for $t:=\alpha-\sum_{1\leq j\leq d}x_j\pmod{1}$  there exists $s\in \Sigma$ such that 
	$$
	\norm{c_js^j-x_j}_\T\leq \frac{\ve}{d} \, \text{ for } \, j\neq j_0  \quad 
	\text{ and  } 	\quad \norm{c_{j_0}s^{j_0}-\alpha+\sum_{1\leq j\leq d,\,  j\neq j_0}x_j}_\T\leq \frac{\ve}{d}.
	$$ 
	Then 
	$$
	\norm{P(s)-\alpha}_\T\leq \sum_{1\leq j\leq d,\,  j\neq j_0}\norm{c_js^j-x_j}_\T+ \norm{c_{j_0}s^{j_0}-\alpha+\sum_{1\leq j\leq d,\,  j\neq j_0}x_j}_\T\leq \ve, 
	$$
	completing the proof. 
\end{proof}

\subsection{A key density property on $\T^d$} The main goal in this subsection is to prove Proposition~\ref{P:dense}. In its proof we use  the following  fact:
\begin{lemma}\label{L:tdense}
	For every $d\in \N$ and $u_1,\ldots, u_d\in \R\setminus\{0\}$, the set 
	$$
	\{(tu_1+k_1,t^2u_2,\ldots, t^du_d+k_d)\colon t\in (0,+\infty), k_1,\ldots, k_d\in \Z\}
	$$
	is dense  in $\R^d$.
\end{lemma}
\begin{proof}
	We can assume that $u_1>0$.
	First, note that if we let $t':= t u_1$ and $u_j':=u_j/u_1^j$, for $j=2,\ldots, d$, we see that it is sufficient to verify the stated property when $u_1=1$. 
	We work with this assumption from now on.  
	
	Let   $x:=(x_1,\ldots, x_d)\in \R^d$. It suffices to show that for every $\ve>0$ there exists $t>0$ such that 
	\begin{equation}\label{E:xdense}
		\norm{t-x_1}_\T\leq \ve, \, \norm{t^2u_2-x_2}_\T\leq \ve, \ldots, \, \norm{t^du_d-x_d}_\T\leq \ve. 
	\end{equation}

	So let $\ve>0$. First, we claim that there exists $\delta_0\in (0,\ve)$ such that every non-trivial integer combination of the polynomials (in $k$)
	$$
	(k+x_1+\delta_0)^2u_2,\ldots, \, (k+x_1+\delta_0)^du_d
	$$
	has at least one irrational non-constant coefficient.  Indeed, note that modulo rational multiples, the coefficients of the (degree $1$) monomial $k$ in these polynomials are respectively 
	$$
	(x_1+\delta_0)u_2,  \, (x_1+\delta_0)^2u_3,\ldots,\, (x_1+\delta_0)^{d-1} u_d.
	$$
	So it suffices to choose $\delta_0$ such that all non-trivial integer combinations of these numbers are irrational.  Such a choice of $\delta_0$ exists, since for any choice of $\ell_1,\ldots, \ell_d\in \Z$, not all of  them  $0$, the equation
	$$
	\ell_1u_2\delta+\ell_2u_3\delta^2+\cdots+\ell_du_d\delta^{d-1}=0
	$$
	has finitely many solutions in the variable $\delta$, so there exists $\delta_0\in (0,\ve)$ such that $x_1+\delta_0$ avoids all  these (countably many) solutions. This proves the claim.

	For this choice of $\delta_0$,  Weyl's criterion~\cite{W16}  gives  that the sequence
	$$
	\big((k+x_1+\delta_0)^2u_2,\ldots, (k+x_1+\delta_0)^du_d\big)_{k\in \N} 
	$$
	is  equidistributed in $\T^d$ and thus dense in $\T^d$. It follows that there exists an integer $k_0>|x_1|+\delta_0$  such that 
	$$
	\norm{(k_0+x_1+\delta_0)^2-x_2}_\T\leq\ve, \ldots, \norm{(k_0+x_1+\delta_0)^d-x_d}_\T\leq\ve.
	$$
	Letting $t:=k_0+x_1+\delta_0$, and recalling that $\delta_0\in (0,\ve)$,  we deduce  that \eqref{E:xdense} holds, completing the proof. 
\end{proof}

\begin{proposition}
	\label{P:dense}
	Let $d\in \N$,  $\Sigma$ be a non-lacunary semigroup of $\N$, and 
	$A$ be a closed 
	subset of $\T^d$,  invariant under  
	$$
	\{(s,s^2,\ldots, s^d)\colon s\in \Sigma\}.
	$$   
	If $0$ is a non-isolated point of $A$, then $\{te_j\colon t\in \T\}\subset A$
	for some   $j\in \{1,\ldots, d\}$.
\end{proposition}
\begin{proof}
	For $d=1$,  this follows from  Furstenberg's Theorem (see Section~\ref{SS:DensityFW}). 
	For general $d\in \N$, we argue as follows using  ideas from~\cite{F} and \cite{Y}, with the caveat that 	we are not working with the  Euclidean norm but with the quantity defined in \eqref{E:norm}. 
	
	Let $\pi\colon\R^d\to\T^d=\R^d/\Z^d$ be the natural projection and let $B=\pi\inv(A)=A+\Z^d$. Then for every $s\in\Sigma$, the closed subset $B$ of $\R^d$ is invariant under the map $x\mapsto (s,s^2,\ldots, s^d)  \cdot x$, and the point $0$ of $\R^d$ is a non-isolated point of $B$. 
	
	For $x=(x_1,\ldots, x_d)\in \R^d$ let 
	\begin{equation}\label{E:norm}
		\nnorm{x}:=|x_1|+|x_2|^{\frac{1}{2}}+\cdots+ |x_d|^{\frac{1}{d}}.
	\end{equation}
	Note that $\nnorm{\cdot}$ satisfies   the triangle inequality and we have the following  identity
	\begin{equation}\label{E:normid}
		\nnorm{(tx_1,t^2x_2,\ldots, t^dx_d)}=|t|\nnorm{(x_1,x_2,\ldots, x_d)}.
	\end{equation}
	Note also that a sequence of vectors in $\R^d$ converges to a vector $x$ with respect to the Euclidean norm if and only if it converges to $x$ with respect to the distance associated with  $\nnorm{\cdot}$;  we use this fact without further reference.
	
	For every $\ve>0$, we let
	\begin{equation}\label{E:De}
		D_\ve:=\Bigl\{ \Big(\frac{x_1}{\nnorm{x}}, \frac{x_2}{\nnorm{x}^2},\ldots, \frac{x_d}{\nnorm{x}^d}\Big)
		\colon x\in B,\ 0<\nnorm{x}\leq\ve\Bigr\}
	\end{equation}
	and
	\begin{equation}\label{E:D}
		D:=\bigcap_{\ve >0}\overline{D_\ve}.
	\end{equation}
	Then $D$ is a closed subset of the compact set  $\{x\in\R^d\colon \nnorm{x}=1\}$.
	Since, by assumption, $0$ is a non-isolated point of $A$ (this is the only point where this assumption is used), for every $\ve>0$ the compact set $\overline{D_\ve}$ is non-empty, so the set $D$ is non-empty.

	We claim that for every $u\in D$ and every $t>0$ we have $(t,t^2,\ldots, t^d)\cdot u\in B$ (a similar statement holds for every $t\in \R$ but we do not need this). 
	So let $u=(u_1,\ldots, u_d)\in D$, $t>0$, and $\varepsilon>0$. 
	It suffices to show that  there exists $b\in B$ with 
	\begin{equation}\label{E:needed} 
		\nnorm{(t,t^2,\ldots, t^d)\cdot u-b}\leq \ve.
	\end{equation}

	Let $\Sigma=\{s_1<s_2<\cdots \}$. Since $\Sigma$ is  a non-lacunary semigroup of $\N$,  we have $\lim_{n\to\infty} s_n=\infty$ and  $\lim_{n\to\infty}\frac{s_{n+1}}{s_n}=1$.
	Hence, for every $\delta>0$  (to be chosen later, depending on $\ve$ and $t$ only) there exists $M=M(\delta)\geq s_1$ such that if $s_{n+1}>M$ we have 
	$$
	| s_{n+1}^j-s_n^j|\leq \delta s_n^j, \quad \text{for } j=1,\ldots, d. 
	$$
	It follows that  for every real $\sigma\geq M$, there exists $s\in\Sigma$ with 
	\begin{equation}\label{E:ssigma}
		|s^j-\sigma^j|\leq \delta \cdot \sigma^j, \quad \text{ for } j=1,\ldots, d.
	\end{equation}
	Indeed,  we can choose $s:=s_{n_0}$ where $n_0\in \N$ satisfies $s_{n_0}\leq \sigma< s_{n_0+1}$, then  $s_{n_0+1}>M$ and we have 	$|\sigma^j-s_{n_0}^j|\leq |s_{n_0+1}^j-s_{n_0}^j|\leq \delta s_{n_0}^j\leq \delta\sigma^j$.
	
	Since $u$ belongs to $D$,  by~\eqref{E:De} and~\eqref{E:D} there exists $ x=(x_1,\ldots, x_d)\in B$, depending on $\ve, t,u,M$,  with
	\begin{equation}\label{E:2}
		0< \nnorm{x}\leq \frac{t}{M} \quad  \text{and} \quad \nnorm{ \big(\frac{x_1}{\nnorm{x}}, \frac{x_2}{\nnorm{x}^2},\ldots, \frac{x_d}{\nnorm{x}^d}\big)-u}\leq \frac{\ve}{2t}.
	\end{equation}

	Applying \eqref{E:ssigma}  with $\delta:= \frac{\ve^d}{2^d(1+t^d)}  $ and    $\sigma:=\frac{t}{\nnorm{x}}\geq M$ (by \eqref{E:2}),  it follows  that there exists $s\in\Sigma$  (depending on $\varepsilon, t, x$) such that 
	\begin{equation}\label{E:3}
		\Bigl|s^j-\frac{t^j}{\nnorm{x}^j}\Bigr|\leq  \delta \cdot\frac{t^j}{\nnorm{x}^j} =\frac{\ve^d}{2^d\nnorm{x}^j}, \quad \text{ for } j=1,\ldots, d.
	\end{equation}

	If we add and subtract $\big(\frac{t x_1}{\nnorm{x}},\frac{t^2x_2}{\nnorm{x}^2},\ldots, \frac{t^dx_d}{\nnorm{x}^d} \big)$ in the expression below and  use the triangle inequality,  we get 
	\begin{multline}\label{E:ts}
	\nnorm{(tu_1,t^2u_2,\ldots, t^du_d)-(sx_1,s^2x_2,\ldots, s^dx_d))}\leq \\
\nnorm{\Big(t\big(\frac{ x_1}{\nnorm{x}}-u_1\big),t^2\big(\frac{ x_2}{\nnorm{x}^2}-u_2\big),\ldots, t^d\big(\frac{ x_d}{\nnorm{x}^d}-u_d\big) \Big)}+\\ 		
\nnorm{\big(x_1\big(\frac{t}{\nnorm{x}}-s\big),x_2\big(\frac{t^2}{\nnorm{x}^2}-s^2\big),
		\ldots, x_d\big(\frac{t^d}{\nnorm{x}^d}-s^d\big) \big)} \leq\ve,
\end{multline}
	where the last estimate follows from \eqref{E:2} and \eqref{E:3} as follows: By \eqref{E:normid} the first term   is  equal to 
	$$
	t \, \large\nnorm{ \big(\frac{x_1}{\nnorm{x}}, \frac{x_2}{\nnorm{x}^2},\ldots, \frac{x_d}{\nnorm{x}^d}\big)-u}\leq \frac{\varepsilon}{2},
	$$
	where the last estimate follows from  \eqref{E:2}.
	By \eqref{E:norm} the second term is equal to 
	$$
	|x_1| \Big|\frac{t}{\nnorm{x}}-s\Big| + |x_2|^{\frac{1}{2}} \Big|\frac{t^2}{\nnorm{x}^2}-s^2\Big|^{\frac{1}{2}}+\cdots+
	|x_d|^{\frac{1}{d}} \Big|\frac{t^d}{\nnorm{x}^d}-s^d\Big|^{\frac{1}{d}},
	$$
	and using  \eqref{E:3} we can  bound this by 
	$$
	\frac{\ve}{2\nnorm{x}} ( |x_1| + |x_2|^{\frac{1}{2}}+\cdots+ |x_d|^{\frac{1}{d}} )=\frac{\ve}{2}.
	$$
	It follows from \eqref{E:ts} that \eqref{E:needed} holds for  $b:=(s,s^2,\ldots, s^d)\cdot  x$, which is an element of  $B$ since $x\in B$ and $B$ is invariant under the map 
	$x\mapsto (s,s^2,\ldots, s^d) \cdot  x$. This  completes the  proof of the claim.
	
	Let $u$ now be any element in $D$. 
	We have just shown that  the set  $C:=\{(t,t^2,\ldots, t^d)\cdot u\colon t>0\}$ is contained  in $B$ and thus the closure  of its image $\pi(C)$ under the natural projection $\pi\colon \R^d\to \T^d$ is contained in $A=\pi(B)$. 
	Since $u\in D$ we have $\nnorm{u}=1$, hence there exists $j_0\in \{1,\ldots, d\}$ such that $u_{j_0}\neq 0$.  We deduce from this and  Lemma~\ref{L:tdense} that   $\{te_{j_0}\colon t\in \R\}\subset C+\Z^d\subset B$, hence $\{te_{j_0}\colon t\in \T\}\subset A $. This completes the proof. 
\end{proof}
We deduce the following:
\begin{corollary}
	\label{C:line}
	Let $d\in \N$,  $\Sigma$ be a non-lacunary semigroup of $\N$, and 
	$A$ be a closed 
	subset of $\T^d$ that is invariant under   
	$$
	\{(s,s^2,\ldots, s^d)\colon s\in \Sigma\}.
	$$   
	If $x$ is a non-isolated  rational point of $A$,  with denominators relatively  prime to $\Sigma$, then 
	$\{x+te_j\colon t\in \T\}\subset A$ 	for some   $j\in \{1,\ldots, d\}$. 
\end{corollary}
\begin{proof}
	As in the proof of Proposition~\ref{prop:periodic} we let  $m$ be an integer that is  relatively prime to every element of $\Sigma$ and such that the coordinates of  $x$ remain invariant when multiplied by any element of the  non-lacunary semigroup $\Sigma_m:=\Sigma\cap (m\Z+1)$. Then the set $A-x$ is closed,  invariant under 
	$$
	\{(s,s^2,\ldots, s^d)\colon s\in \Sigma_m\},
	$$ 
	and has $0$ as a non-isolated point.  Applying  Proposition~\ref{P:dense}  with $\Sigma_m$ instead of $\Sigma$, we get  that 
	there exists  $j\in \{1,\ldots, d\}$ such that  $\{te_j\colon t\in \T\}\subset A-x$, completing the proof. 
\end{proof}
\subsection{Proof of  Proposition~\ref{P:densehigher}}
	Since $A$ is an  infinite subset of $\T^d$, the set $A'$ of non-isolated points of $A$ is non-empty. It is also  closed and invariant under the set  $\Sigma^{1,\ldots, d}$ of Proposition~\ref{prop:periodic}. Applying  Proposition~\ref{prop:periodic} for $\ell_1:=1,\ldots, \ell_d:=d$,  we get that $A'$ contains a rational point with denominators   relatively prime  to $\Sigma$.  
	The result now follows from Corollary~\ref{C:line} and the fact that $A'\subset A$ (since $A$ is closed).

\end{document}